\documentclass[12pt,reqno]{amsart}
\usepackage{amssymb}
\usepackage[latin1]{inputenc}
\usepackage{dsfont}
\usepackage{hyperref}
\usepackage{fullpage}
\newtheorem{theorem}{Theorem}[section]

\newtheorem{proposition}[theorem]{Proposition}
\newtheorem{corollary}[theorem]{Corollary}
\theoremstyle{definition}
\newtheorem{definition}[theorem]{Definition}

\theoremstyle{remark}
\newtheorem{remark}[theorem]{Remark}
\numberwithin{equation}{section}
\begin{document}
\title{A remarkable $\sigma$-finite measure associated with last passage times and penalisation problems}
\author[J. Najnudel]{Joseph Najnudel}
\address{Institut f\"ur Mathematik, Universit\"at Z\"urich, Winterthurerstrasse 190,
8057-Z\"urich, Switzerland}
\email{\href{mailto:joseph.najnudel@math.uzh.ch}{joseph.najnudel@math.uzh.ch}}
\author[A. Nikeghbali]{Ashkan Nikeghbali}
\email{\href{mailto:ashkan.nikeghbali@math.uzh.ch}{ashkan.nikeghbali@math.uzh.ch}}
 \dedicatory{Dedicated
to Eckhard Platen for his 60th birthday}

\date{\today}
\begin{abstract}
In this paper, we give a global view of the results we have obtained in relation with a remarkable class 
of submartingales, called $(\Sigma)$, and which are stated in \cite{NN1}, \cite{NN2}, \cite{NN3} and \cite{NN4}.
More precisely, we associate to a given submartingale in this class $(\Sigma)$, defined on a filtered probability
space $(\Omega, \mathcal{F}, \mathbb{P}, (\mathcal{F}_t)_{t \geq 0})$, satisfying some technical conditions, 
a $\sigma$-finite measure $\mathcal{Q}$ on $(\Omega, \mathcal{F})$, such that for all 
$t \geq 0$, and for all events $\Lambda_t
 \in \mathcal{F}_t$: $$ \mathcal{Q} [\Lambda_t, g\leq t] = \mathbb{E}_{\mathbb{P}} [\mathds{1}_{\Lambda_t} X_t]$$
where $g$ is the last hitting time of zero of the process $X$. This measure $\mathcal{Q}$ has already been 
defined in several particular cases, some of them are involved in the study of Brownian penalisation, 
and others are related with problems in mathematical finance. More precisely, the existence of $\mathcal{Q}$ 
in the general case solves a problem stated by D. Madan, B. Roynette and M. Yor, in a paper studying the link
between Black-Scholes formula and last passage times of certain submartingales. Once the measure $\mathcal{Q}$
is constructed, we define a family of nonnegative martingales, corresponding to 
the local densities (with respect to $\mathbb{P}$) of the finite measures which are absolutely continuous 
with respect to $\mathcal{Q}$. We study in detail the relation between $\mathcal{Q}$ and this class of martingales,
and we deduce a decomposition of any nonnegative martingale into three parts, corresponding to the 
decomposition of finite measures on $(\Omega, \mathcal{F})$ as the sum of three measures, such that the first one
is absolutely continuous with respect to $\mathbb{P}$, the second one is absolutely continuous with respect to 
$\mathcal{Q}$ and the third one is singular with respect to $\mathbb{P}$ and $\mathcal{Q}$. This decomposition 
can be generalized to supermartingales. Moreover, if under $\mathbb{P}$, the process $(X_t)_{t \geq 0}$ is a diffusion
satisfying some technical conditions, one can state a penalisation result involving the measure $\mathcal{Q}$, and 
 generalizing a theorem given in \cite{NRY}. Now, in the construction of the measure 
$\mathcal{Q}$, we encounter the following problem:
 if $(\Omega,\mathcal{F},(\mathcal{F}_t)_{t \geq 0},\mathbb{P})$
 is a filtered probability space satisfying
 the usual assumptions, then it is usually not possible to extend to $\mathcal{F}_{\infty}$ 
(the $\sigma$-algebra generated by $(\mathcal{F}_t)_{t \geq 0}$) a coherent family
 of probability measures $(\mathbb{Q}_t)$ indexed by $t \geq 0$, each
of them being defined on $\mathcal{F}_t$. That is why we must not assume the usual assumptions in our case. 
On the other hand, the 
usual assumptions are crucial in order to obtain the existence of regular versions of paths (typically
 adapted and continuous or adapted and c\`adl\`ag versions) for most stochastic processes of interest, such
 as the local time of the standard Brownian motion, stochastic integrals, etc.  In order to fix this problem, we introduce another augmentation of filtrations, intermediate between the right continuity
and the usual conditions, and call it N-augmentation in this paper. This augmentation has also been  considered by Bichteler \cite{Bi}.  
Most of the important results of the theory of stochastic processes which are generally proved under
 the usual augmentation still
 hold under the N-augmentation; moreover this new augmentation allows the extension of a coherent family 
of probability measures whenever this is possible with the original filtration. 
\end{abstract}
\maketitle
%\tableofcontents
\section*{Notation}
In this paper, $(\Omega,\mathcal{F},(\mathcal{F}_t)_{t \geq 0},\mathbb{P})$ will denote a filtered
 probability space. $\mathcal{C}(\mathbb{R}_+,\mathbb{R})$ is the space of continuous functions 
from $\mathbb{R}_+$ to $\mathbb{R}$. $\mathcal{D}(\mathbb{R}_+,\mathbb{R})$ is the space of c\`adl\`ag functions
 from $\mathbb{R}_+$ to $\mathbb{R}$. If  $Y$ is a random variable, we denote indifferently by
 $\mathbb{P}[Y]$ or by $\mathbb{E}_{\mathbb P}[Y]$
the expectation of $Y$ with respect to $\mathbb{P}$.
\section{Introduction}
This paper reviews some recent results obtained by Cheridito, Nikeghbali and Platen \cite{CNP} and by the authors of this paper \cite{NN1,NN2,NN3,NN4} on the last zero of some remarkable stochastic processes and its relation with a universal $\sigma$-finite measure which has many remarkable properties and which seems to be in many places the key object to interpret in a unified way results which do not seem to be related at first sight. The last zero of c\`adl\`ag and adapted processes will play an essential role in our discussions: this fact is quite surprising since such a random time is not a stopping time and hence falls outside the domain of applications of the classical theorems in stochastic analysis. 

The problem originally came from a paper of Madan, Roynette and Yor \cite{MRY} on the pricing of European put options where they are able to represent the price of a European put option in terms of the probability distribution of some last passage time. More precisely, they prove that if $(M_t)_{t \geq 0}$ is a continuous nonnegative local martingale  defined on a filtered probability space $(\Omega,\mathcal{F},(\mathcal{F}_t)_{t \geq 0},\mathbb{P})$ satisfying the usual assumptions, and such that  $\lim_{t\to\infty}M_t=0$, then 
\begin{equation}\label{BS}
	(K-M_t)^+=K\mathbb{P}(g_K\leq t|\mathcal{F}_t)
	\end{equation}where $K\geq0$ is a constant and $g_K=\sup\{t \geq 0: M_t=K\}$.  Formula (\ref{BS}) tells that it is enough to know the terminal value of the submartingale $(K-M_t)^+$ and its last zero $g_K$ to reconstruct it.  Yet a nicer interpretation of  (\ref{BS}) is suggested in  \cite{MRY}:  there exists a measure $\mathcal{Q}$, a random time $g$, such that the submartingale $X_t=(K-M_t)^+$ satisfies
\begin{equation}\label{masterequation}
\mathcal{Q} \left[ F_t \, \mathds{1}_{g \leq t} \right] = \mathbb{E} \left[F_t X_t \right],	
\end{equation}
for any $t\geq0$ and for any  bounded $\mathcal{F}_t$-measurable random variable $F_t$. Indeed, it easily follows from (\ref{BS}) that, in this case,  $\mathcal{Q}=K. \mathbb{P}$ and $g=g_K$.  It is also clear that if a stochastic process $X$ satisfies  (\ref{masterequation}), then it is a submartingale. The problem of finding the class of submartingales which satisfy (\ref{masterequation}) is posed in \cite{MRY}:\\

\textbf{Problem 1 (\cite{MRY}):}  for which nonnegative submartingales $X$ can we find a $\sigma$-finite measure $\mathcal{Q}$ and the end of an optional set $g$ such that
\begin{equation}\label{masterequation2}
\mathcal{Q} \left[ F_t \, \mathds{1}_{g \leq t} \right] = \mathbb{E} \left[F_t X_t \right]?
\end{equation}

It is also noticed in \cite{MRY} that other instances of formula (\ref{masterequation}) have already been discovered: for example, in \cite{AY1}, Az\'ema and Yor proved that for any continuous and uniformly integrable martingale $M$, (\ref{masterequation2}) holds for $X_t=|M_t|$, $\mathcal{Q}=|M_\infty|.\mathbb{P}$ and $g=\sup\{t \geq 0: M_t=0\}$, or equivalently $$|M_t|=\mathbb{E}[|M_\infty|\mathds{1}_{g\leq t}|\mathcal{F}_t].$$ Here again the measure $\mathcal{Q}$ is finite. Recently,  another particular case where the measure $\mathcal{Q}$ is not finite  was obtained by  Najnudel, Roynette and Yor
in their study of Brownian penalisation (see \cite{NRY}). They prove the existence of the measure $\mathcal{Q}$ when $X_t=|Y_t|$ is the absolute value of the canonical process $(Y_t)_{t\geq0}$ under the Wiener measure $\mathbb{W}$, on the space $\mathcal{C}(\mathbb{R}_+,\mathbb{R})$ equipped with the filtration generated by the canonical process. In this case, the measure $\mathcal{Q}$, which we shall denote hereafter $\mathcal{W}$ to remind we are working on the Wiener space, is not finite but  $\sigma$-finite and is singular with respect to the Wiener measure: it satisfies $\mathcal{W}(g=\infty)=0$, where $g=\sup\{t \geq 0:\;X_t=0\}$. However, a closer look at this last example reveals that Problem 1 may lead to some paradox. Indeed, the existence of the measure $\mathcal{Q}$ implies that the filtration $(\mathcal{F}_t)_{t\geq0}$  should not satisfy the usual assumptions. Indeed, if this were the case, then for all $t \geq 0$, the event $\{g >
 t\}$ would have probability one (under $\mathbb{W}$) and then, would be in $\mathcal{F}_0$ and, a fortiori, in $\mathcal{F}_t$. 
If one assumes that $\mathcal{W}$ exists, this implies:
$$\mathcal{W} [g > t, g \leq t] = \mathbb{E}_{\mathbb{W}} [\mathds{1}_{g > t} \, X_t],$$
and then
$$ \mathbb{E}_{\mathbb{W}}[X_t] =  0,$$
which is absurd. But now, if one does not complete the original probability space, then it is possible to show (see Section \ref{sec::2}) that there does not exist a continuous and adapted version of the local time for the Wiener process $Y$, which is one of the key processes in the study of Brownian penalisations! More generally, one cannot apply most of the useful  results from the general theory of stochastic processes such as the existence of c\`adl\`ag versions for martingales, the Doob-Meyer decomposition, the d\'ebut theorem, etc. Consequently, one has to provide conditions on the underlying filtered probability space $(\Omega,\mathcal{F},(\mathcal{F}_t)_{t \geq 0},\mathbb{P})$ under which not only equation (\ref{masterequation2}) in Problem 1 can hold, but also under which the existence of regular versions (e.g. continuous and adapted version for the Brownian local time) of stochastic processes of interest do exist:\\

\textbf{Problem 2} What are the natural conditions to impose to $(\Omega,\mathcal{F},(\mathcal{F}_t)_{t \geq 0},\mathbb{P})$ in order to have (\ref{masterequation2})  and at the same time the existence of regular versions for stochastic processes of interest?\\

In fact we shall see that the existence of the measure $\mathcal{Q}$ is related to the problem of extension of a  coherent family of probability measures: given a coherent family of probability measures $(\mathbb{Q}_t)_{t\geq0}$ with $\mathbb{Q}_t$ defined on $\mathcal{F}_t$ (i.e.: the restriction of $\mathbb{Q}_t$ to $\mathcal{F}_s$ is $\mathbb{Q}_s$ for $t\geq s$), does there exist a probability measure $\mathbb{Q}_{\infty}$ defined on $\mathcal{F}_{\infty}=\bigvee_{t\geq0}\mathcal{F}_t$ such that the restriction of $\mathbb{Q}_\infty$ to $\mathcal{F}_t$ is $\mathbb{Q}_t$? The solution to such extension problems is well-known and very well detailed in the book by Parthasarathy \cite{Parth}, however, it is surprising that such a fundamental problem (if one thinks about the widely used changes of probability measures which are only locally absolutely continuous with respect to a reference probability measure) has rarely been considered together with the existence of c\`adl\`ag versions for martingales, the Doob-Meyer decomposition, the d\'ebut theorem, etc. 

In Section \ref{sec::2}, we shall propose an alternative  augmentation of filtrations, intermediate
 between the right-continuous version and the usual augmentation, under which most of the properties generally
proved under usual conditions (existence of c\`adl\`ag versions, existence of Doob-Meyer decomposition, etc.)
 are preserved, but for which it is still possible to extend compatible families of probability measures. We proposed this augmentation in \cite{NN2} and discovered later than Bichteler has also proposed it in his book \cite{Bi}.
By using this  augmentation, we are able to prove the existence of the measure $\mathcal{Q}$ under very
general assumptions. The relevant class of submartingales is called $(\Sigma)$, and the precise conditions
under which we can show that $\mathcal{Q}$ exists are stated in Section \ref{sec::3}. Just before this statement, 
we give a detailed solution of Problem 1 in the particular case where $\mathcal{Q}$ is absolutely continuous
with respect to $\mathbb{P}$, with some applications to financial modeling.

The measure $\mathcal{Q}$ has some remarkable properties which we shall detail in Section  \ref{sec::4}. One of its most striking properties is that it allows a unified treatment of many problems of penalisation on the Wiener space which do not seem to be related at first sight. The framework which is generally used it the following (see the book \cite{NRY} for more details and references): we consider $\mathbb{W}$, the Wiener measure on $\mathcal{C}(\mathbb{R}_+,\mathbb{R})$, endowed with its natural filtration $(\mathcal{F}_s)_{
s \geq 0}$, $(\Gamma_t)_{t \geq 0}$ a family of 
nonnegative random variables on the same space, such that 
$$0 < \mathbb{W} [\Gamma_t]  < \infty,$$
and for $t \geq 0$, the probability measure $$\mathbb{Q}_t := 
\frac{\Gamma_t}{\mathbb{W}[\Gamma_t]} \, . \mathbb{W}.$$
Under these assumptions, Roynette, Vallois and Yor have 
proven (see \cite{RVY}) that for many examples of family of functionals $(\Gamma_t)_{t \geq 0}$, there exists a probability 
measure $\mathbb{Q}_{\infty}$ which can be considered as the weak limit of $(\mathbb{Q}_t)_{t \geq 0}$ 
when $t$ goes to infinity, in the following sense: for all $s \geq 0$ and for all bounded, 
$\mathcal{F}_s$-measurable random variables $F_s$, one has
$$\mathbb{Q}_t [F_s] \underset{t \rightarrow \infty}{\longrightarrow} \mathbb{Q}_{\infty} [F_s].$$
For example, the measure $\mathbb{Q}_{\infty}$ exists for the following families of functionals\footnote{The discussion around Problem 2 shows that the results described below are not correct because a continuous and adapted version of the Brownian local time does not exist with the natural filtration. The conditions given in Section \ref{sec::2} will remedy this gap.}
$(\Gamma_t)_{t \geq 0}$:
\begin{itemize}
\item $\Gamma_t = \phi(L_t)$, where $(L_t)_{t \geq 0}$ is the local time at zero of the canonical process $X$, 
and $\phi$ is a nonnegative, integrable function from $\mathbb{R}_+$ to $\mathbb{R}_+$.
\item $\Gamma_t = \phi(S_t)$, where $S_t$ is the supremum of $X$ on the interval $[0,t]$, and $\phi$ is, again, 
a nonnegative, integrable function from $\mathbb{R}_+$ to $\mathbb{R}_+$. 
\item $\Gamma_t = e^{\lambda L_t + \mu |X_t|}$, where $(L_t)_{t \geq 0}$ is, again, the local time at zero 
of $X$. 
\end{itemize}
In \cite{NRY}, Najnudel, Roynette and Yor obtain a result which gives the existence of $\mathbb{Q}_{\infty}$
for a  large class of families of functionals $(\Gamma_t)_{t \geq 0}$. The proof 
of this penalisation result involves, in an essential way, the $\sigma$-finite measure $\mathcal{W}$ on
the space $\mathcal{C}(\mathbb{R}_+,\mathbb{R})$ described above. More precisely, they 
prove that for a relatively large class of functionals $(\Gamma_t)_{t \geq 0}$, 
\begin{equation}\label{ntm}
\mathbb{Q}_\infty=\dfrac{\Gamma_\infty}{\mathcal{W}[\Gamma_\infty]}. \mathcal{W},
\end{equation}
\noindent 
where $\Gamma_{\infty}$ is the limit of $\Gamma_t$ when $t$ goes to infinity, which is supposed to 
exist everywhere. \\

\textbf{Problem 3} Can we use our general existence theorem on the measure $\mathcal{Q}$ in Section \ref{sec::3} to extend the general result on the Brownian penalisation problem and $\mathcal{W}$ to a larger class of stochastic processes?\\

At end of Section \ref{sec::4}, we shall see that the answer to Problem 3 is positive (which extends (\ref{ntm})),
if under $\mathbb{P}$, the submartingale $(X_t)_{t \geq 0}$ is a diffusion satisfying some technical
conditions. In particular, our result can be applied to suitable powers of
Bessel processes if the dimension is in the interval $(0,2)$ (which includes
 the case of the reflected Brownian motion). Unlike all our other results, for which we are able
 to get rid of the Markov and scaling properties that have been used 
so far in the Brownian studies, the Markov property plays here a crucial role.

\section{A new kind of augmentation of filtrations consistent with the problem of extension of measures}\label{sec::2}
The discussion in this section follows closely our paper \cite{NN2}; in particular the proofs which are not provided can be found there.
\subsection{Understanding the problem}
In stochastic analysis, most of the interesting properties of continuous time random processes cannot be established if 
one does not  assume that their trajectories satisfy some  regularity conditions. For example,
a nonnegative c\`adl\`ag martingale converges almost surely, but if the c\`adl\`ag assumption
is removed, the result becomes false in general. Recall a very simple counter-example: on the filtered
probability space $\big((\mathcal{C}( \mathbb{R}_+, \mathbb{R}), \mathcal{F}, (\mathcal{F}_t)_{t \geq 0}
, \mathbb{W} \big)$, where $\mathcal{F}_t = \sigma \{ X_s, 0 \leq s \leq t\}$, 
$\mathcal{F} = \sigma \{X_s, s \geq 0\}$, $(X_s)_{s \geq 0}$ is the canonical process and $ \mathbb{W}$ the Wiener measure, the 
martingale 
$$\big(M_t := \mathds{1}_{X_t = 1}\big)_{t \geq 0},$$
which is a.s. equal to zero for each fixed $t \geq 0$, does not converge at infinity. 
That is the reason  why one generally  considers a c\`adl\`ag version of  a martingale. 
However there are fundamental examples of stochastic processes for which such a version does not exist. Indeed
let us define on the filtered probability
space $\big((\mathcal{C}( \mathbb{R}_+, \mathbb{R}), \mathcal{F}, (\mathcal{F}_t)_{t \geq 0}
, \mathbb{W} \big)$ described above, the stochastic process $(\mathcal{L}_t)_{t \geq 0}$ as follows:
$$\mathcal{L}_t  = \Phi \left(\underset{m \rightarrow \infty}{\lim \inf} \int_0^t f_m(X_s) ds \right),$$
where $f_m$ denotes the density of a centered Gaussian variable with variance $1/m$ and
$\Phi$ is the function from $\mathbb{R}_+ \cup \{\infty\}$ to $\mathbb{R}_+$ such that 
$\Phi(x) = x$ for $x < \infty$ and $\Phi(\infty) = 0$. The process $(\mathcal{L}_t)_{t \geq 0}$ 
is a version of the local time of the canonical process at level zero, which is defined everywhere
and $(\mathcal{F}_t)_{t \geq 0}$-adapted. It is known that  the process:
$$\big(M_t := |X_t| - \mathcal{L}_t \big)_{t \geq 0}$$
is an $(\mathcal{F}_t)_{t \geq 0}$-martingale. However, $(M_t)_{t \geq 0}$ does not admit 
a c\`adl\`ag version which is adapted. In other words, there exists no c\`adl\`ag, adapted
version $(L_t)_{t \geq 0}$ for the local time at level zero of the canonical process! This 
property can be proved in the following way: let us consider an Ornstein-Uhlenbeck process
 $(U_t)_{t \geq 0}$, starting from zero, and
let us define the process $(V_t)_{t \geq 0}$ by:
$$V_t = (1-t)U_{t/(1-t)}$$
for $t < 1$, and 
$$V_t = 0$$
for $t \geq 1$. This process is a.s. continuous: we denote by $\mathbb{Q}$ its distribution.
One can check the following properties:
\begin{itemize}
\item For all $t \in [0,1)$, the restriction of $\mathbb{Q}$ to $\mathcal{F}_t$ is absolutely continuous
with respect to the corresponding restriction of $\mathbb{W}$. 
\item Under $\mathbb{Q}$, $\mathcal{L}_t\to\infty$  a.s. when $ t\to1,\;t < 1$. 
\end{itemize}
\noindent
By the second property, the set $\{ \mathcal{L}_t \underset{t \rightarrow 1, t<1}{\longrightarrow} \infty\}$
 has probability one under $\mathbb{Q}$. Since it is negligible under $\mathbb{P}$, it is essential 
to suppose that it is not contained in $\mathcal{F}_0$, if we need to have the first property: the filtration
must not be completed. 
The two properties above imply
\begin{align*}
\mathbb{Q} \left[ L_{1 - 2^{-n}}  \underset{n \rightarrow \infty}{\longrightarrow} \infty \right]
& \geq \mathbb{Q} \left[ \mathcal{L}_{1 - 2^{-n}}  \underset{n \rightarrow \infty}{\longrightarrow} \infty,
\, \forall n \in \mathbb{N}, \, L_{1-2^{-n}} = \mathcal{L}_{1 - 2^{-n}}  \right] \\ 
& \geq 1 - \sum_{n \in \mathbb{N}} \mathbb{Q} \left[L_{1-2^{-n}} \neq \mathcal{L}_{1 - 2^{-n}} \right] = 1. 
\end{align*}
\noindent
The last equality is due to the fact that for all $n \in \mathbb{N}$,
$$ \mathbb{W} \left[L_{1-2^{-n}} \neq \mathcal{L}_{1 - 2^{-n}} \right] = 0,$$
and then
$$\mathbb{Q}  \left[L_{1-2^{-n}} \neq \mathcal{L}_{1 - 2^{-n}} \right] = 0,$$
since $L_{1-2^{-n}}$ and $\mathcal{L}_{1-2^{-n}}$ are $\mathcal{F}_{1-2^{-n}}$-measurable 
and since the restriction of $\mathbb{Q}$ to this $\sigma$-algebra is absolutely continuous with respect to
$\mathbb{W}$. We have thus proved that there exist some paths such that $L_{1 - 2^{-n}}$ tends
to infinity with $n$, which contradicts the fact  that $(L_t)_{t \geq 0}$ is c\`adl\`ag.
From this we also deduce that in general there do not exist c\`adl\`ag versions for martingales. Similarly many other
important results from stochastic analysis cannot be proved
 on the most general filtered probability space, e.g. the existence of the Doob-Meyer decomposition 
for submartingales and the d\'ebut theorem (see for instance \cite{DM} and \cite{DM2}).
In order to avoid this technical problem, it is generally assumed that the filtered probability 
space on which the processes are constructed satisfies the usual conditions, i.e. the filtration 
is complete and right-continuous. \\

But now, if we wish to perform a change of probability measure (for example,
by using the Girsanov theorem), this assumption reveals to be too restrictive. Let us illustrate this fact by a simple example. Let us consider the filtered probability 
space $\big(\mathcal{C}(\mathbb{R}_+, \mathbb{R}), \widetilde{\mathcal{F}},
 (\widetilde{\mathcal{F}}_t)_{t \geq 0}, \widetilde{\mathbb{W}}\big)$
 obtained, from the Wiener space 
$\big(\mathcal{C}(\mathbb{R}_+, \mathbb{R}), \mathcal{F}, (\mathcal{F}_t)_{t \geq 0}, \mathbb{W}\big)$
described above, by taking its usual augmentation, i.e.:
\begin{itemize}
\item $\widetilde{\mathcal{F}}$ is the $\sigma$-algebra generated by $\mathcal{F}$ and its negligible sets.
\item For all $t \geq 0$, $\widetilde{\mathcal{F}}_t$ is $\sigma$-algebra generated by 
$\mathcal{F}_t$ and the negligible sets of $\mathcal{F}$.
\item $\widetilde{\mathbb{W}}$ is the unique possible extension of $\mathbb{W}$ to the completed
$\sigma$-algebra $\widetilde{\mathcal{F}}$.
\end{itemize}
\noindent
Let us also consider the family of probability measures $(\mathbb{Q}_t)_{t \geq 0}$, 
such that $\mathbb{Q}_t$ is defined on $\widetilde{\mathcal{F}}_t$ by
$$\mathbb{Q}_t = e^{X_t - \frac{t}{2} }. \, \widetilde{\mathbb{W}}_{ \, |\widetilde{\mathcal{F}_t} }.$$
This family of probability measures is coherent, i.e. for $0 \leq s \leq t$, the restriction of $\mathbb{Q}_t$ to 
$\widetilde{\mathcal{F}}_s$ is equal to $\mathbb{Q}_s$. However, unlike what one would expect,
 there does not exist a probability measure 
$\mathbb{Q}$ on $\widetilde{\mathcal{F}}$ such that its restriction to 
$\widetilde{\mathcal{F}}_s$ is equal to $\mathbb{Q}_s$ for all $s \geq 0$. Indeed, let us assume
that $\mathbb{Q}$ exists. 
The event
$$A := \{ \forall t \geq 0, X_t \geq -1 \} $$
satisfies $\widetilde{\mathbb{W}} [A] = 0$, 
and then $A \in \widetilde{\mathcal{F}}_0$ by completeness, which implies that $\mathbb{Q}[A] = 0$.
On the other hand, under $\mathbb{Q}_t$, for all $t \geq 0$, the process $(X_s)_{0 \leq s \leq t}$ is a Brownian motion
with drift 1, and hence under $\mathbb{Q}$. One
deduces that:
$$\mathbb{Q} [ \forall s \in [0,t], \, X_s \geq -1 ]
= \widetilde{\mathbb{W}} [\forall s \in [0,t], \, X_s \geq -s-1 ] 
\geq \widetilde{\mathbb{W}} [\forall s \geq 0, \, X_s \geq -s-1 ].$$
Consequently, by letting $t$ go to infinity, one obtains:
$$\mathbb{Q} [A] \geq \widetilde{\mathbb{W}} [\forall s \geq 0, \, X_s \geq -s-1 ] > 0,$$
which is a contradiction. Therefore, the usual conditions are not suitable  for the problem of extension of coherent 
probability measures.  In fact one can observe that the argument above does not depend on the completeness of $\widetilde{\mathcal{F}}$, but
only on the fact that $\widetilde{\mathcal{F}}_0$ contains all the sets in $\widetilde{\mathcal{F}}$ of 
probability zero. That is why it still remains available if we consider, with the notation above, the 
space $\big(\mathcal{C}(\mathbb{R}_+, \mathbb{R}), \mathcal{F}, (\mathcal{F}'_t)_{t \geq 0}, \mathbb{W}\big)$,
where for all $t \geq 0$, $\mathcal{F}'_t$ is the $\sigma$-algebra generated by $\mathcal{F}_t$ and 
the sets in $\mathcal{F}$ of probability zero. 
 
\subsection{The N-usual augmentation}
In order to make compatible the general results on the extension of probability measures problem and the existence of regular versions for stochastic processes, we propose an augmentation which is intermediate between right continuity and the usual augmentation, and we call it the \textit{N-augmentation}. As already mentioned, Bichteler has introduced this augmentation before us in his book \cite{Bi}, and has called it the \emph{natural augmentation}. But when we discovered this augmentation, we were not aware of Bichteler's work and this is reflected in the difference in our approaches. Hence the interested reader would benefit by looking at both \cite{Bi} and \cite{NN2}.
\begin{definition}[\cite{NN2}] \label{negligible}
Let $(\Omega, \mathcal{F}, (\mathcal{F}_t)_{t \geq 0}, \mathbb{P})$ be a filtered probability space.
A subset $A$ of $\Omega$ is N-negligible with respect to the space 
$(\Omega,\mathcal{F}, (\mathcal{F}_t)_{t \geq 0}, \mathbb{P})$, iff there exists a sequence $(B_n)_{n \geq 0}$
of subsets of $\Omega$, such that for all $n \geq 0$, $B_n \in \mathcal{F}_n$, $\mathbb{P} [B_n] = 0$, and
$$A \subset \bigcup_{n \geq 0} B_n.$$  
\end{definition}
\begin{remark}
The integers do not play a crucial r\^ole in Definition \ref{negligible}. If 
$(t_n)_{n \geq 0}$ is an unbounded sequence in $\mathbb{R}_+$, one
can replace the condition $B_n \in \mathcal{F}_n$ by the condition $B_n \in \mathcal{F}_{t_n}$.  
\end{remark}
\noindent
Let us now define a notion which is the analog of completeness for N-negligible sets. It 
is the main ingredient in the definition of what we shall call the N-usual conditions:  
\begin{definition}[\cite{NN2}] \label{complete}
A filtered probability space $(\Omega,\mathcal{F}, (\mathcal{F}_t)_{t \geq 0}, \mathbb{P})$,
  is N-complete iff all the N-negligible sets of this space are contained in $\mathcal{F}_0$. It
satisfies the N-usual conditions iff it is N-complete and the filtration 
$(\mathcal{F}_t)_{t \geq 0}$ is right-continuous.
\end{definition}
\noindent
It is natural to ask if from a given filtered probability 
space, one can define in a canonical way a space which satisfies the N-usual conditions and 
which is as "close" as possible to the initial space. The answer to this question is positive in the following
sense: 
\begin{proposition}[\cite{NN2}]
Let $(\Omega, \mathcal{F}, (\mathcal{F}_t)_{t \geq 0}, \mathbb{P})$ be a filtered probability space, and 
$\mathcal{N}$ the family of its N-negligible sets. 
Let $\widetilde{\mathcal{F}}$ be the $\sigma$-algebra generated by $\mathcal{N}$ and 
$\mathcal{F}$, and for all $t \geq 0$, 
 $\widetilde{\mathcal{F}}_t$  the $\sigma$-algebra generated by $\mathcal{N}$ and
$\mathcal{F}_{t+}$,
where $$ \mathcal{F}_{t+} := \bigcap_{u > t} \mathcal{F}_t.$$ Then there exists a unique probability measure 
$\widetilde{\mathbb{P}}$ on $(\Omega, \widetilde{\mathcal{F}})$ which coincides with 
$\mathbb{P}$ on $\mathcal{F}$, and the space $(\Omega, \widetilde{\mathcal{F}}, 
 (\widetilde{\mathcal{F}}_t)_{t \geq 0},
\widetilde{\mathbb{P}})$ satisfies the N-usual conditions. Moreover, if 
 $(\Omega, \mathcal{F}', (\mathcal{F}'_t)_{t \geq 0}, \mathbb{P}')$ is a filtered probability space
satisfying the N-usual conditions, such that $\mathcal{F}'$ contains $\mathcal{F}$,  $\mathcal{F}'_t$ contains
$\mathcal{F}_t$ for all $t \geq 0$, and if $\mathbb{P}'$ is an extension of $\mathbb{P}$, then 
$\mathcal{F}'$ contains $\widetilde{\mathcal{F}}$, $\mathcal{F}'$ contains $\widetilde{\mathcal{F}}_t$, 
for all $t \geq 0$ and $\mathbb{P}'$ is an extension of $\widetilde{\mathbb{P}}$.
In other words,  $(\Omega, \widetilde{\mathcal{F}}, 
 (\widetilde{\mathcal{F}}_t)_{t \geq 0},
\widetilde{\mathbb{P}})$ is the smallest extension of $(\Omega, \mathcal{F}, 
(\mathcal{F}_t)_{t \geq 0}, \mathbb{P})$ which satisfies the N-usual conditions:
we call it the N-augmentation of $(\Omega, \mathcal{F}, 
(\mathcal{F}_t)_{t \geq 0}, \mathbb{P})$ 
\end{proposition}

Once the N-usual conditions are defined, it is natural to compare them with the usual conditions. 
One has the following result:
\begin{proposition}
Let $(\Omega, \mathcal{F}, (\mathcal{F}_t)_{t \geq 0}, \mathbb{P})$ be a filtered probability
space which satisfies the N-usual conditions. Then for all $t \geq 0$, the space 
$(\Omega, \mathcal{F}_t, (\mathcal{F}_s)_{0 \leq 
s \leq t}, \mathbb{P})$ satisfies the usual conditions. 
\end{proposition}
\begin{proof}
The right-continuity of $(\mathcal{F}_s)_{0 \leq s \leq t}$ is obvious, let us prove the completeness. 
If $A$ is a negligible set of $(\Omega, \mathcal{F}_t, \mathbb{P})$, there exists 
$B \in \mathcal{F}_t$, such that 
$A \subset B$ and $\mathbb{P}[B]=0$. One deduces immediately that $A$ is N-negligible with respect 
to $(\Omega, \mathcal{F}, (\mathcal{F}_t)_{t \geq 0}, \mathbb{P})$, and by N-completeness
of this filtered probability space, $A \in \mathcal{F}_0$. 
\end{proof}
\noindent
This relation between the usual conditions and the N-usual conditions is the main ingredient to prove that one can replace the usual conditions
by the N-usual conditions in most of the classical results in stochastic calculus. For example, the following can be proved:
\begin{itemize}
\item If $(X_t)$ is a martingale, then it admits a c\`adl\`ag modification, which is unique 
up to indistinguishability. 
\item If $(\Omega, \mathcal{F}, (\mathcal{F}_t)_{t \geq 0}, \mathbb{P})$ is a filtered probability space
satisfying the N-usual conditions, and if $(X_t)_{t \geq 0}$ is an adapted process defined on this space such that 
 there exists a c\`adl\`ag version (resp. continuous version) $(Y_t)_{t \geq 0}$
 of $(X_t)_{t \geq 0}$, then there exists a c\`adl\`ag and adapted version (resp. continuous and adapted version) of 
$(X_t)_{t \geq 0}$, which is necessarily indistinguishable from $(Y_t)_{t \geq 0}$.
\item Let $(\Omega, 
\mathcal{F}, (\mathcal{F}_t)_{t \geq 0}, \mathbb{P})$ be a filtered probability space satisfying the
N-usual conditions, 
and let $A$ be a progressive subset of $\mathbb{R}_+ \times \Omega$. Then the d\'ebut of $A$, i.e. 
the random time $D(A)$ such that for all $\omega \in \Omega$:
$$D(A)(\omega) := \inf \{t \geq 0, (t,\omega) \in A \}$$
is an $(\mathcal{F}_t)_{t \geq 0}$-stopping time. 
\item Let $(X_t)_{t \geq 0}$ be a right-continuous submartingale defined on a filtered probability
space $(\Omega, \mathcal{F}, (\mathcal{F}_t)_{t \geq 0}, \mathbb{P})$, satisfying the N-usual conditions.
 We suppose that $(X_t)_{t \geq 0}$ is of class $(DL)$, i.e. for all $a \geq 0$, $(X_T)_{T \in \mathcal{T}_a}$
is uniformly integrable, where $\mathcal{T}_a$ is the family of the $(\mathcal{F}_t)_{t \geq 0}$-stopping
times which are bounded by $a$ (for example, every nonnegative submartingale is of class $(DL)$).
Then, there exist a right-continuous $(\mathcal{F}_t)_{t \geq 0}$-martingale $(M_t)_{t \geq 0}$ and 
an increasing process
 $(A_t)_{t \geq 0}$ starting at zero, such that: $$X_t = M_t + A_t$$ for all $t \geq 0$, and for every
bounded, right-continuous martingale $(\xi_s)_{s \geq 0}$,
$$\mathbb{E} \left[\xi_t A_t \right] =\mathbb{E} \left[ \int_{(0,t]} \xi_{s-} dA_s \right],$$
where $\xi_{s-}$ is the left-limit of $\xi$ at $s$, almost surely well-defined for all $s > 0$.
The processes $(M_t)_{t \geq 0}$ and $(A_t)_{t \geq 0}$ are uniquely determined, up to indistinguishability. 
Moreover, they can be chosen to be continuous if $(X_t)_{t \geq 0}$ is a continuous process. 
\item The section theorem holds in a filtered probability space satisfying the N-usual assumptions and hence optional projections are well-defined.
\end{itemize}
\subsection{Extension of measures and the N-usual augmentation}
We first state a well-known Parthasarathy  type condition for the probability measures extension problem.

\begin{definition} \label{P}
Let $(\Omega, \mathcal{F}, (\mathcal{F}_t)_{t \geq 0})$ be a filtered measurable space, such that
$\mathcal{F}$ is the $\sigma$-algebra generated by $\mathcal{F}_t$, 
$t \geq 0$: $\mathcal{F}=\bigvee_{t\geq0}\mathcal{F}_t$. We shall say that the
 property (P) holds if and only if $(\mathcal{F}_t)_{t \geq 0}$ 
enjoys the following conditions: 
\begin{itemize}
\item For all $t \geq 0$, $\mathcal{F}_t$ is generated by a countable number of sets;
\item For all $t \geq 0$, there exist a Polish space $\Omega_t$, and a surjective map 
 $\pi_t$ from $\Omega$ to $\Omega_t$, such that $\mathcal{F}_t$ is the $\sigma$-algebra of the inverse
 images, by $\pi_t$, of Borel sets in $\Omega_t$, and such that for all $B \in \mathcal{F}_t$, 
 $\omega \in \Omega$, $\pi_t (\omega) \in \pi_t(B)$ implies $\omega \in B$;
\item If $(\omega_n)_{n \geq 0}$ is a sequence of elements of $\Omega$, such that for all $N \geq 0$,
$$\bigcap_{n = 0}^{N} A_n (\omega_n) \neq \emptyset,$$
where $A_n (\omega_n)$ is the intersection of the sets in $\mathcal{F}_n$ containing $\omega_n$, 
then:
$$\bigcap_{n = 0}^{\infty} A_n (\omega_n) \neq \emptyset.$$
\end{itemize}
\end{definition}
\noindent
Given this technical definition, one can state the following result:
\begin{proposition}[\cite{Parth}, \cite{NN2}] \label{extension}
Let $(\Omega, \mathcal{F}, (\mathcal{F}_t)_{t \geq 0})$ be a filtered measurable space satisfying the 
property (P), and let, for $t \geq 0$, $\mathbb{Q}_t$ be a probability measure on $(\Omega, \mathcal{F}_t)$, 
such that for all $t \geq s \geq 0$, $\mathbb{Q}_s$ is the restriction of $\mathbb{Q}_t$ to $\mathcal{F}_s$.
Then, there exists a unique measure $\mathbb{Q}$ on $(\Omega, \mathcal{F})$ such that for all $t \geq 0$,
its restriction to $\mathcal{F}_t$ is equal to $\mathbb{Q}_t$. 
\end{proposition}
\noindent
One can easily deduce the following corollary which is often used in practice:
\begin{corollary}
Let 	 $\Omega$ be $\mathcal{C}(\mathbb{R}_+,\mathbb{R}^d)$, the space of continuous functions from 
	$\mathbb{R}_+$ to $\mathbb{R}^d$, or $\mathcal{D}(\mathbb{R}_+,\mathbb{R}^d)$, the space of c\`adl\`ag functions
 from $\mathbb{R}_+$ 
	to $\mathbb{R}^d$ (for some $d \geq 1$). For $t \geq 0$, define $(\mathcal{F}_t)_{t \geq 0}$
as the natural filtration of the canonical process $Y$, 
and $\mathcal{F}=\bigvee_{t\geq0}\mathcal{F}_t$. Then  $(\Omega, \mathcal{F},(\mathcal{F}_t)_{t \geq 0})$
 satisfies property (P).

\end{corollary}
\begin{proof}
  Let us prove
this result for c\`adl\`ag functions (for continuous functions, the result is similar and
  proved in \cite{SV}). For all $t \geq 0$, $\mathcal{F}_t$ is 
generated by the variables $Y_{rt}$, for $r$, rational, in $[0,1]$, hence, it is countably generated.
For the second property, one can take for $\Omega_t$, the set of c\`adl\`ag functions
from $[0,t]$ to $\mathbb{R}^d$, and for $\pi_t$, the restriction to the interval $[0,t]$.
The space $\Omega_t$ is Polish if one endows it with the Skorokhod metric. Moreover, 
its Borel $\sigma$-algebra is equal to the $\sigma$-algebra generated by the coordinates, a result 
from which one easily  deduces the properties of $\pi_t$ which need to be satisfied. The 
third property is easy to check: let us suppose that $(\omega_n)_{n \geq 0}$ is a sequence
 of elements of $\Omega$, such that for all $N \geq 0$,
$$\bigcap_{n = 0}^{N} A_n (\omega_n) \neq \emptyset,$$
where $A_n (\omega_n)$ is the intersection of the sets in $\mathcal{F}_n$ containing $\omega_n$.
Here,  $A_n (\omega_n)$ is the set of functions $\omega'$ which coincide with 
$\omega_n$ on $[0,n]$. Moreover, for $n \leq n'$, integers, the intersection
of $A_n (\omega_n)$ and $A_{n'} (\omega_{n'})$ is not empty, and then $\omega_n$ and $\omega_{n'}$ coincide 
on $[0,n]$. Therefore, there exists a c\`adl\`ag function $\omega$ which coincides
with $\omega_n$ on $[0,n]$, for all $n$, which implies:
$$\bigcap_{n = 0}^{\infty} A_n (\omega_n) \neq \emptyset.$$
\end{proof}
\begin{remark}
It is easily seen that the conditions of Proposition \ref{extension} are not satisfied
 by the space $\mathcal{C}([0,1],\mathbb{R})$ endowed with the filtration $(\mathcal{F}_t)_{t \geq 0}$, 
where $\mathcal{F}_t$ is the $\sigma$-algebra generated by the canonical process up to time $t/(1+t)$.
An explicit counter example is provided in \cite{FI}.
\end{remark}
The next proposition shows how condition (P) combines with the N-usual augmentation:
\begin{proposition}[\cite{NN2}]
Let $(\Omega, \mathcal{F}, (\mathcal{F}_t)_{t \geq 0}, \mathbb{P})$ be the N-augmentation of 
a filtered probability space satisfying the property (P). Then if $(\mathbb{Q}_t)_{t \geq 0}$ 
is a coherent family of probability measures, $\mathbb{Q}_t$ defined on $\mathcal{F}_t$, and absolutely 
continuous with respect to the restriction of $\mathbb{P}$ to $\mathcal{F}_t$, there exists
a unique probability measure $\mathbb{Q}$ on $\mathcal{F}$ which coincides with $\mathbb{Q}_t$ on 
$\mathcal{F}_t$, for all $t \geq 0$. 
\end{proposition}
\section{A universal $\sigma$-finite measure $\mathcal{Q}$} \label{sec::3}
\subsection{The class $(\Sigma)$}
We now have all the ingredients to rigorously answer Problem 1 raised in the Introduction. For this we will first  need to introduce a special class of local submartingales which was first introduced by Yor \cite{Y} and further studied by Nikeghbali \cite{N} and Cheridito, Nikeghbali and Platen \cite{CNP}.

\begin{definition} 
Let $(\Omega,\mathcal{F},(\mathcal{F}_t)_{t \geq 0},\mathbb{P})$ be a filtered probability space.
 A nonnegative (local) submartingale $(X_t)_{t \geq 0}$ is of class $(\Sigma)$, if it can
 be decomposed as
$X_t = N_t + A_t$ where $(N_t)_{t \geq 0}$ and $(A_t)_{t \geq 0}$ are $(\mathcal{F}_t)_{t \geq 0}$-adapted 
processes satisfying the following assumptions:
\begin{itemize}
\item $(N_t)_{t \geq 0}$ is a c\`adl\`ag (local) martingale.
\item $(A_t)_{t \geq 0}$ is a continuous increasing process, with $A_0 = 0$.
\item The measure $(dA_t)$ is carried by the set $\{t \geq 0, X_t = 0 \}$.
\end{itemize}
We shall say that $(X_t)_{t\geq0}$ is of class $(\Sigma D)$ if $X$ is of class $(\Sigma)$ and of class $(D)$.
\end{definition}

The class $(\Sigma)$ contains many well-known examples of stochastic processes (see e.g. \cite{N}) such as nonnegative local martingales, $|M_t|$, $M_t^+$, $M_t^-$ if $M$ is a continuous local martingale, the drawdown process $S_t-M_t$ where $M$ is a local martingale with only negative jumps and $S_t=\sup_{u\leq t}M_u$, the relative drawdown process $1-\dfrac{M_t}{S_t}$ if $M_0\neq0$, the age process of the standard Brownian motion $W_t$ in the filtration of the zeros of the Brownian motion, namely $\sqrt{t-g_t}$, where $g_t=\sup\{u\leq t: W_u=0\}$, etc. 
Moreover one notes that if $X$ is of $(\Sigma)$, then $X_t+M_t$ is also of class $(\Sigma)$ for any strictly positive c\`adl\`ag local martingale $M$.
Another key property of the class $(\Sigma)$ is the following stability result which follows from an application of Ito's formula and a monotone class argument. 
\begin{proposition}[\cite{N}] \label{lemmatrans}
Let $(X_t)_{t \geq 0}$ be of class $(\Sigma)$  and let $f : \mathbb{R} \to \mathbb{R}$ be
a locally bounded Borel function. Let us further assume that $(\Omega,\mathcal{F},(\mathcal{F}_t)_{t \geq 0},\mathbb{P})$ satisfies the usual assumptions or is N-complete. Denote $F(x) = \int_0^x f(y) dy$.
Then the process $(f(A_t) X_t)_{t \geq 0}$ is again of class $(\Sigma)$ with decomposition:
\begin{equation} \label{decomp}
f(A_t)X_t = f(0) X_0 + \int_0^t f(A_u) dN_u + F(A_t).
\end{equation}
\end{proposition}

\subsection{A special case related to financial modeling}\label{sec::3.1}
Now we state a theorem which gives sufficient conditions under which a process of the class $(\Sigma)$ which
 converges to $X_\infty$ a.s. satisfies (\ref{masterequation2}). This result is an extension of a result by Az\'ema-Yor \cite{AY1} and Az\'ema-Meyer-Yor \cite{AMY}. Indeed, in the case when $X$ is of class $(\Sigma D)$, one could deduce it from part 1 of Theorem 8.1 in  \cite{AMY}. The proof is very simple in this case and we give it.
\begin{theorem}[\cite{CNP}] \label{thmrepL}
Assume that $(\Omega,\mathcal{F},(\mathcal{F}_t)_{t \geq 0},\mathbb{P})$ satisfies the usual assumptions or is N-complete. Let  $(X_t)_{t \geq 0}$ be a process of class $(\Sigma)$ such that $\lim_{t\to\infty}X_t=X_\infty$ exists
 a.s. and is finite (in particular $N_\infty$ and $A_\infty$ exist and are a.s. finite). Let	$$
	g := \sup \{t : X_t = 0\} \quad \mbox{with the convention }
	\sup \emptyset = 0.
	$$
\begin{enumerate}
	\item If $(X_t)_{t \geq 0}$ is of class $(D)$, then
	\begin{equation} \label{repL}
	X_T = \mathbb{E}[X_{\infty} 1_{\{g \leq T\}} \mid \mathcal{F}_T] \quad
	\mbox{for every stopping time } T.
\end{equation}
\item More generally, if there exists a strictly
 positive Borel function $f$
 such that $(f(A_t)X_t)_{t \geq 0}$ is of class $(D)$, then (\ref{repL}) holds.
	\item If $(N_t^+)_{t \geq 0}$ is of class $(D)$, then 
	(\ref{repL}) holds.
\end{enumerate}
 \end{theorem}

\begin{proof}
(1) For a given stopping time $T$, denote
$$
d_T = \inf \{t > T : X_t = 0\} \quad \mbox{with the convention } \inf \emptyset = \infty.
$$
One checks that $d_T$ is a stopping time. Since $X_{\infty} 1_{\{g \le T\}} = X_{d_T}$ and 
$A_T = A_{d_T}$, it follows from Doob's optional stopping theorem that
$$
\mathbb{E}[ X_{\infty} 1_{\{g \le T\}} \mid {\mathcal F}_T]
= \mathbb{E}[N_{d_T} + A_{d_T} \mid {\mathcal F}_T]
= \mathbb{E}[N_{d_T} + A_T \mid {\mathcal F}_T]
= N_T + A_T = X_T.
$$

(2) Assume that there exists a strictly positive Borel function such that
 $(f(A_t)X_t)_{t \geq 0}$ is of
 class $(D)$. This property is preserved if one replaces $f$ by a smaller strictly positive Borel function, 
hence, one can suppose that 
$f$ is locally bounded. Then $(f(A_t)X_t)_{t \geq 0}$ is of class $(\Sigma D)$, and from part (1) of the theorem, we have:
$$f(A_T)X_T=\mathbb{E}[f(A_\infty)X_{\infty} 1_{\{g \leq T\}} \mid \mathcal{F}_T].$$ But 
on the set $\{g \leq T\}$, we have $A_\infty=A_T$, and consequently 
$$f(A_T)X_T=f(A_T)\mathbb{E}[X_{\infty} 1_{\{g \leq T\}} \mid \mathcal{F}_T].$$ The result
 follows by dividing both sides by $f(A_T)$ which is strictly positive. 

(3) Since $X\geq0$ and since $(N_t^+)_{t \geq 0}$ is of class $(D)$, we note that $(\exp(-A_t)X_t)_{t \geq 0}$ is of
 class $(D)$ and the result follows from (2).
\end{proof}
The above theorem is used in \cite{CNP} for financial applications. For example, if $(M_t)_{t \geq 0}$ is
 a nonnegative local martingale, converging to $0$ (which is a reasonable
 model for stock prices or for portfolios under the benchmark approach), then the drawdown process
 $DD_t=\max_{u\leq t} M_u-M_t$ is of class $(\Sigma)$ as well as the relative drawdown 
process $rDD_t=1-\frac{M_t}{\max_{u\leq t} M_u}$. For example, if $(M_t)_{t \geq 0}$ is a strict nonnegative
 local martingale such as the inverse of the $3$ dimensional Bessel process, then $(DD_t)_{t \geq 0}$ is
 of class $(\Sigma)$, satisfies condition (3) of the above theorem but is not uniformly integrable.  Note
that Theorem \ref{thmrepL} applies in this case, even if the local martingale 
part of the process $(DD_t)_{t \geq 0}$, which is equal to $(-M_t)_{t \geq 0}$, is not a true martingale. 
\subsection{The general case} \label{sec::3.2}
A general solution for Problem 1 is provided in \cite{NN1}; the measure $\mathcal{W}$ is a 
special case of this theorem. For the general case, we need to be very careful: the theorem below
 would be wrong under the usual assumptions and it would also be wrong if the
 filtration $(\mathcal{F}_t)_{t \geq 0}$ does not allow the extension of coherent 
probability measures. The N-usual assumptions are here to ensure in particular that there exists a continuous and adapted version of the process $A$ which is  defined everywhere (think of $A_t$ being for example the local time at the level $0$ of the Wiener process). According to Section \ref{sec::2}, typical probability spaces where the following theorem holds are $\mathcal{C}(\mathbb{R}_+,\mathbb{R})$ and $\mathcal{D}(\mathbb{R}_+,\mathbb{R})$.
\begin{theorem} \label{all}
Let $(X_t)_{t \geq 0}$ be a true submartingale of the class $(\Sigma)$: its local martingale part 
$(N_t)_{t \geq 0}$ is a true martingale, and $X_t$ is integrable for all $t \geq 0$.
We suppose that $(X_t)_{t \geq 0}$ is defined on a filtered probability space 
$(\Omega, \mathcal{F}, \mathbb{P}, (\mathcal{F}_t)_{t \geq 0})$ satisfying the property (NP), in particular, this
space satisfies the N-usual conditions and $\mathcal{F}$ is the $\sigma$-algebra generated by $\mathcal{F}_t$
 for $t \geq 0$. 
Then, there exists a unique $\sigma$-finite measure $\mathcal{Q}$, defined on 
$(\Omega, \mathcal{F}, \mathbb{P})$, such that for $g:= \sup\{t \geq 0, X_t = 0 \}$:
\begin{itemize}
\item $\mathcal{Q} [g = \infty] = 0$;
\item For all $t \geq 0$, and for all $\mathcal{F}_t$-measurable, bounded random variables $\Gamma_t$,
$$\mathcal{Q} \left[ \Gamma_t \, \mathds{1}_{g \leq t} \right] = \mathbb{P} \left[\Gamma_t X_t \right].$$
\end{itemize}
\noindent
\end{theorem}
\noindent
In \cite{NN1}, the measure $\mathcal{Q}$ is explicitly constructed, in the following way (with a
slightly different notation). Let $f$ be
a Borel, integrable, strictly positive and bounded function from $\mathbb{R}$ to $\mathbb{R}$, and 
let us define the function $G$ by the formula:
$$ G(x) = \int_{x}^{\infty} f(y) \, dy.$$
One can prove
that the process \begin{equation}
 \left( M^f_t := G(A_t) - \mathbb{E}_{\mathbb{P}} [G(A_{\infty})| \mathcal{F}_t] +
f(A_t) X_t \right)_{t \geq 0}, \label{MTF2}
\end{equation}
\noindent
is a martingale with respect to $\mathbb{P}$ and the filtration $(\mathcal{F}_t)_{t \geq 0}$.
Since $(\Omega, \mathcal{F}, \mathbb{P}, (\mathcal{F}_t)_{t \geq 0})$ satisfies the N-usual conditions
and since $G(A_t) \geq G(A_{\infty})$, one can suppose that this martingale is nonnegative and c\`adl\`ag, by 
choosing carefully the 
version of $\mathbb{E}_{\mathbb{P}} [G(A_{\infty})|\mathcal{F}_t]$. In this case, since 
$(\Omega, \mathcal{F}, \mathbb{P}, (\mathcal{F}_t)_{t \geq 0})$ satisfies the property (NP), there 
exists a unique finite measure $\mathcal{M}^f$ such that for all $t \geq 0$, and for all bounded,
 $\mathcal{F}_t$-measurable functionals $\Gamma_t$: 
$$\mathcal{M}^f [\Gamma_t] = \mathbb{E}_{\mathbb{P}} [\Gamma_t M_t^f].$$
Now, since $f$ is strictly positive, one can define a $\sigma$-finite measure $\mathcal{Q}^f$ by:
$$\mathcal{Q}^f := \frac{1}{f(A_{\infty})} \, . \mathcal{M}^f.$$
It is proved is \cite{NN1} that if the function $G/f$ is unformly bounded (this condition 
is, for example, satisfied for $f (x) = e^{-x}$), then $\mathcal{Q}^f$ satisfies the conditions
defining $\mathcal{Q}$ in Theorem \ref{all}, which implies the existence part of this result. 
The uniqueness part is proved just after in a very easy way: one remarkable consequence 
of it is the fact that $\mathcal{Q}^f$ does not depend on the choice of $f$.
The measure $\mathcal{Q}$ has many other interesting properties, which will be detailed 
in Section \ref{sec::4}. 
\section{Further properties of $\mathcal{Q}$ and some remarkable associated martingales, and penalisation
results} \label{sec::4}
The properties of $\mathcal{Q}$ given in this section are stated and proved in \cite{NN3}, except the 
results concerning penalisation, which are shown in \cite{NN4}.  
By the construction of $\mathcal{Q}$ described above, it is clear that if $f$ is a 
 Borel, integrable, strictly positive and bounded function from $\mathbb{R}$ to $\mathbb{R}$, and if, 
with the notation above, $G/f$ is uniformly bounded, then
\begin{equation}
\mathcal{M}^f = f(A_{\infty}) \, . \mathcal{Q}. \label{MFQ}
\end{equation}
Now, by using (\ref{MTF2}), it is possible to construct $(M_t^f)_{t \geq 0}$, and then $\mathcal{M}^f$,
for all Borel, integrable and nonnegative functions $f$. Moreover, it is proved in \cite{NN3} that 
(\ref{MFQ}) remains true for any function $f$ satisfying these weaker assumptions.  
The relation between the functional $f(A_{\infty})$ and the martingale $(M_t^f)_{t \geq 0}$ can be 
generalized as follows:
\begin{proposition} \label{MtF}
We suppose that the assumptions of Theorem \ref{all} hold, and we take the same notation. 
Let $F$ be a $\mathcal{Q}$-integrable, nonnegative functional defned on $(\Omega, \mathcal{F})$. Then, there 
exists
 a c\`adl\`ag $\mathbb{P}$-martingale $(M_t(F))_{t \geq 0}$ such
 that the measure $\mathcal{M}^F := F. \mathcal{Q}$ is the unique finite measure
 satisfying, for all $t \geq 0$, and 
for all bounded, $\mathcal{F}_t$-measurable functionals $\Gamma_t$:
$$\mathcal{M}^F [\Gamma_t] = \mathbb{P} [\Gamma_t M_t(F)].$$
The martingale $(M_t(F))_{t \geq 0}$ is unique up to indistinguishability.
\end{proposition}
\noindent
If $f$ is Borel, integrable and nonnegative, then $f(A_{\infty})$ is integrable with respect to $\mathcal{Q}$ 
and one has
$$M_t (f(A_{\infty})) = M_t^f,$$ 
which gives an explicit expression for $M_t (f(A_{\infty}))$. This explicit form can be generalized 
to the martingale $(M_t(F))_{t \geq 0}$ for all nonnegative and $\mathcal{Q}$-integrable 
functionals $F$, if the submartingale $(X_t)_{t \geq 0}$ is uniformly integrable. More precisely, one 
has the following result:
\begin{proposition} \label{uniformlyintegrable}
Let us suppose that the assumptions of Theorem \ref{all} are satisfied, and that the process $(X_t)_{t \geq 0}$ is
uniformly integrable. Then, $X_t$ tends a.s. to a limit $X_{\infty}$ when $t$ goes to infinity, and 
the measure $\mathcal{Q}$ is absolutely continuous with respect to $\mathbb{P}$, with density $X_{\infty}$. 
Moreover, a nonnegative functional $F$ is integrable with respect to $\mathcal{Q}$ iff $F X_{\infty}$ is
 integrable
with respect to $\mathbb{P}$, in this case, $(M_t(F))_{t \geq 0}$ a the c\`adl\`ag version of the 
conditional expectation $(\mathbb{P} [F X_{\infty} | \mathcal{F}_t])_{t \geq 0}$. In particular, it 
is uniformly integrable, and it converges a.s. and  in 
$L^1$ to $F X_{\infty}$ when $t$ goes to infinity. 
\end{proposition}
\noindent
A more interesting case is when we suppose that $A_{\infty}$ is infinite, $\mathbb{P}$-almost surely.
From now, until the end of this section, we always implicitly make this assumption (which is satisfied, in 
particular, if $(X_t)_{t \geq 0}$ is a reflected Brownian motion). The following result gives the asymptotic
behaviour of $(X_t)_{t \geq 0}$ under $\mathcal{Q}$, when $t$ goes to infinity, and 
the behaviour of $M_t(F)$, which is not the same under $\mathbb{P}$ and
under $\mathcal{Q}$:
\begin{proposition}
Let us suppose that the assumptions of Theorem \ref{all} are satisfied, and that
 $A_{\infty} = \infty$, $\mathbb{P}$-almost surely. Then, $\mathcal{Q}$-almost everywhere, $X_t$ tends to 
infinity when $t$ goes to infinity, and 
for all nonnegative, $\mathcal{Q}$-integrable
functionals $F$, the martingale $(M_t(F))_{t \geq 0}$ tends $\mathbb{P}$-almost surely to zero and 
$\mathcal{Q}$-almost everywhere to infinity. Moreover, one has:
$$\frac{M_t(F)}{X_t} \underset{t \rightarrow \infty}{\longrightarrow} F,$$
$\mathcal{Q}$-almost everywhere.
\end{proposition}
\noindent
 By definition, the martingales of the form $(M_t(F))_{t \geq 0}$ are exactly the local densities of the 
finite measures which are absolutely continuous with respect to $\mathcal{Q}$. This situation is
similar to the case of uniformly integrable, nonnegative martingales, which are local densities
of finite measures, absolutely continuous with respect to $\mathbb{P}$. The following decomposition of 
nonnegative supermartingales, already proved in \cite{NRY} in the case of 
the reflected Brownian motion, involves simultaneously these two kind of martingales:
\begin{proposition} \label{decomposition}
Let us suppose that the assumptions of Theorem \ref{all} are satisfied, and that
 $A_{\infty} = \infty$, $\mathbb{P}$-almost surely. Let $Z$ be a nonnegative, c\`adl\`ag
 $\mathbb{P}$-supermartingale. 
We denote by $Z_{\infty}$ the $\mathbb{P}$-almost sure limit of $Z_t$ when $t$ goes to infinity. 
Then, $\mathcal{Q}$-almost everywhere, the quotient $Z_t/X_t$ is well-defined for $t$ large enough
 and converges, when $t$ goes to infinity, to a limit $z_{\infty}$, integrable with respect
 to $\mathcal{Q}$,  and $(Z_t)_{t \geq 0}$ decomposes
 as
$$\left(Z_t = M_t (z_{\infty}) + \mathbb{P} [Z_{\infty} |\mathcal{F}_t] + \xi_t \right)_{t \geq 0},$$
where $(\mathbb{P} [Z_{\infty} |\mathcal{F}_t])_{t \geq 0}$ denotes a c\`adl\`ag version of the 
conditional expectation of $Z_{\infty}$ with respect to $\mathcal{F}_t$, and $(\xi_t)_{t \geq 0}$ is 
a nonnegative, c\`adl\`ag $\mathbb{P}$-supermartingale, such that:
\begin{itemize}
\item $Z_{\infty} \in L^1_+ (\mathcal{F}, \mathbb{P})$, hence $\mathbb{P} [Z_{\infty} | \mathcal{F}_t]$ 
converges $\mathbb{P}$-almost surely and in $L^1 (\mathcal{F}, \mathbb{P})$ towards $Z_{\infty}$.
\item $\frac{\mathbb{P} [Z_{\infty} | \mathcal{F}_t] + \xi_t}{X_t} \underset{t \rightarrow \infty}
{\longrightarrow} 0$, $\mathcal{Q}$-almost everywhere.
\item $M_t(z_{\infty}) + \xi_t \underset{t \rightarrow \infty}{\longrightarrow} 0$, $\mathbb{P}$-almost surely.  
\end{itemize}
\noindent
Moreover, the decomposition is unique in the following sense: let $z'_{\infty}$ be a $\mathcal{Q}$-integrable, 
nonnegative functional, $Z'_{\infty}$ a $\mathbb{P}$-integrable, nonnegative random variable,
 $(\xi'_t)_{t \geq 0}$ a c\`adl\`ag, nonnegative $\mathbb{P}$-supermartingale, and let us suppose that
 for all $t\geq 0$,
$$Z_t = M_t (z'_{\infty}) + \mathbb{P} [Z'_{\infty} |\mathcal{F}_t] + \xi'_t.$$
Under these assumptions, if for $t$ going to infinity, $\xi'_t$ tends $\mathbb{P}$-almost surely to zero and
 $\xi'_t /X_t$ tends $\mathcal{Q}$-almost everywhere to zero, then $z'_{\infty} = z_{\infty}$, $\mathcal{Q}$-almost
everywhere, 
$Z'_{\infty} = Z_{\infty}$, $\mathbb{P}$-almost surely, and $\xi'$ is $\mathbb{P}$-indistinguishable with $\xi$. 
\end{proposition}
\noindent
This result implies the following characterisation of martingales of the form $(M_t(F))_{t \geq 0}$:
\begin{corollary}
Let us suppose that the assumptions of Theorem \ref{all} are satisfied, and that
 $A_{\infty} = \infty$, $\mathbb{P}$-almost surely. Then, a c\`adl\`ag, nonnegative $\mathbb{P}$-martingale
$(Z_t)_{t\geq 0}$ is of the form $(M_t(F))_{t \geq 0}$ for a nonnegative, $\mathcal{Q}$-integrable
 functional
$F$, if and only if:
\begin{equation}
\mathbb{P}[Z_0] = \mathcal{Q} \left( \underset{t \rightarrow \infty}{\lim} \, \frac{Z_t}{X_t} \right). \label{pz0}
\end{equation}
Note that, by Proposition \ref{decomposition}, the limit above necessarily exists $\mathcal{Q}$-almost 
everywhere.
\end{corollary}
\noindent
If in Proposition \ref{decomposition}, $(Z_t)_{t \geq 0}$ is a nonnegative martingale, the corresponding 
decomposition can be interpreted as a Radon-Nykodym decomposition of finite measures. Indeed, let us observe that
since the space satisfies the property (NP), there exists a unique finite measure $\mathbb{Q}_Z$ 
on $(\Omega, \mathcal{F})$, such that for all $t \geq 0$, its restriction to $\mathcal{F}_t$ has 
density $Z_t$ with respect to $\mathbb{P}$. If one writes the decomposition
$$Z_t = M_t(z_{\infty}) + \mathbb{P} [Z_{\infty} | \mathcal{F}_t] + \xi_t,$$
one deduces:
$$\mathbb{Q}_Z = z_{\infty} \,. \mathcal{Q} + Z_{\infty} \, . \mathbb{P} + \mathbb{Q}_{\xi},$$
where the restriction of $\mathbb{Q}_{\xi}$ to $\mathcal{F}_t$ has density $\xi_t$ with respect to 
$\mathbb{P}$. In \cite{NN3}, it is proved that $\mathbb{Q}_{\xi}$ is 
singular with respect to $\mathbb{P}$ and $\mathcal{Q}$, hence one has
 a decomposition of $\mathbb{Q}_Z$ into three parts:
\begin{itemize}
\item A part which is absolutely continuous with respect to $\mathbb{P}$.
\item A part which is absolutely continuous with respect to $\mathcal{Q}$. 
\item A part which is singular with respect to $\mathbb{P}$ and $\mathcal{Q}$.
\end{itemize}
\noindent
This decomposition is unique, as a consequence of the uniqueness of the Radon-Nykodym decomposition.

Another interesting problem is to find a relation between the measure $\mathcal{Q}$ and the last passage
times of $(X_t)_{t \geq 0}$ at a given level, which can be different from zero. The following result 
proves this relation when $(X_t)_{t \geq 0}$ is continuous:
\begin{proposition} \label{xta}
Let us suppose that the assumptions of Theorem \ref{all} are satisfied, the submartingale 
$(X_t)_{t \geq 0}$ is continuous and $A_{\infty} = \infty$ almost surely 
under $\mathbb{P}$. For $a \geq 0$, let $g^{[a]}$ be the last hitting time of the 
interval $[0,a]$:
$$g^{[a]} = \sup \{t \geq 0, X_t \leq a \}.$$
Then, the measure $\mathcal{Q}$ satisfies the following formula, available for any $t \geq 0$, and for all
$\mathcal{F}_t$-measurable, bounded variables $\Gamma_t$:
\begin{equation}\mathcal{Q} \left[ \Gamma_t \, \mathds{1}_{g^{[a]} \leq t} \right] = \mathbb{P} \left[\Gamma_t (X_t-a)_+
 \right]. \label{Qa}
\end{equation}
\noindent
Moreover, $((X_t-a)_+)_{t \geq 0}$ is a submartingale of class $(\Sigma)$ and the $\sigma$-finite measure 
obtained by applying 
Theorem \ref{all} to it is equal to $\mathcal{Q}$. 
\end{proposition}
The relation between the measure $\mathcal{Q}$ and the last hitting time of a given level $a$
is one of the main ingredients of the proof of our penalisation result stated below. 
Since we are able to show this statement only when $(X_t)_{t \geq 0}$ is a diffusion satisfying 
some technical conditions, let us give another result, strongly related to penalisations, but available
under more general assumptions:
\begin{proposition}
Let us suppose that the assumptions of Theorem \ref{all} are satisfied, and $A_{\infty}= \infty$, 
$\mathbb{P}$-almost surely. Let $(F_t)_{t \geq 0}$ be a c\`adl\`ag,
 adapted, nonnegative, nonincreasing and uniformly bounded process, such that for some 
$a > 0$, one has for all $t \geq 0$, $X_t = X_{g^{[a]}}$ on the set $\{t \geq g^{[a]}\}$. Then, if 
$F_g$ is $\mathcal{Q}$-integrable, and if one
defines $F_{\infty}$  as the limit of $F_t$ for $t$ going to infinity (in particular, $F_{\infty} = F_{g^{[a]}}$ 
for $g^{[a]} < \infty$), one has:
$$\mathbb{P} [F_tX_t ] \underset{t \rightarrow \infty}{\longrightarrow} \mathcal{Q} [F_{\infty}]. $$
\end{proposition}
\noindent
In \cite{NN3}, there are another version of this result, with slightly different assumptions. 
In order to obtain a penalisation result, we need to estimate the expectation of $F_t$ under 
$\mathbb{P}$, instead of the expectation of $F_tX_t$. This gives an extra factor, depending on 
$t$, in the asymptotics, which justifies some restriction on the process $(X_t)_{t \geq 0}$. 
Let us now describe the precise framework which is considered. 
We define $\Omega$ as the space of continuous functions from $\mathbb{R}_+$ to 
$\mathbb{R}_+$, $(\mathcal{F}^0_t)_{t \geq 0}$ as the natural filtration of $\Omega$, and 
$\mathcal{F}^0$, as the $\sigma$-algebra generated by $(\mathcal{F}^0_t)_{t \geq 0}$. 
The probability $\mathbb{P}^0$, defined on $(\Omega, \mathcal{F}^0)$, satisfies the following condition:
under $\mathbb{P}^0$, the canonical process is a recurrent diffusion in natural scale, starting from a fixed 
point $x_0 \geq 0$, with zero as an instantaneously reflecting barrier, and such that its speed measure 
is absolutely continuous with respect to Lebesgue measure on $\mathbb{R}_+$, with a density 
$m : \mathbb{R}^*_+ \rightarrow \mathbb{R}^*_+$, continuous, and such that $m(x)$ is 
equivalent to $cx^{\beta}$ when $x$ goes to infinity, for some $c > 0$ and $\beta > -1$. Moreover,
we suppose that there exists $C > 0$ such that for all $x > 0$, $m(x) \leq Cx^{\beta}$ if $\beta \leq 0$, and 
$m(x) \leq C(1+x^{\beta})$ if $\beta > 0$.
Let us now define the filtered probability space $(\Omega, \mathcal{F}, (\mathcal{F}_t)_{t \geq 0}, \mathbb{P})$ as
the N-augmentation of $(\Omega, \mathcal{F}^0, (\mathcal{F}^0_t)_{t \geq 0}, \mathbb{P}^0)$:
$(\Omega, \mathcal{F}, (\mathcal{F}_t)_{t \geq 0}, \mathbb{P})$ satisfies property (NP) and
under $\mathbb{P}$, the law of the canonical process is a diffusion with the same parameters as 
under $\mathbb{P}^0$. This diffusion is in natural scale, and one can deduce from this fact and 
the assumptions above that the canonical process $(X_t)_{t \geq 0}$ is a submartingale of class $(\Sigma)$. In 
particular, $X_t$ is integrable
 for all $t \geq 0$. Moreover,  the local time $(L_t)_{t \geq 0}$ of $(X_t)_{t \geq 0}$ at level zero, is its
increasing process. One deduces that Theorem \ref{all} applies, which defines the corresponding 
measure $\mathcal{Q}$, and we can check that we are in the situation where $L_{\infty} = \infty$, 
$\mathbb{P}$-almost surely. Under these assumptions, we can prove the following result: 
\begin{proposition} \label{pen5}
Let $(F_t)_{t \geq 0}$ be a c\`adl\`ag,
 adapted, nonnegative, uniformly bounded and nonincreasing process. We assume that there exists
$a >0$ such that for all $t \geq 0$, $X_t = 
X_{g^{[a]}}$ on the set $\{t \geq g^{[a]}\}$, we
define $F_{\infty}$  as the limit of $F_t$ for $t$ going to infinity, and we suppose
 that $F_{\infty}$ is integrable with respect to $\mathcal{Q}$.
Then, there exists $D>0$ such that for all $s \geq 0$ and for all events $\Lambda_s \in \mathcal{F}_s$:
$$t^{1/(\beta+2)} \, \mathbb{P} [F_t \mathds{1}_{\Lambda_s}] \underset{t \rightarrow \infty}{\longrightarrow}
 D \, \mathcal{Q} [ F_{\infty} \mathds{1}_{\Lambda_s} ].$$
\end{proposition}
\noindent
The following penalisation result is a straightforward consequence of Proposition \ref{pen5}:
\begin{proposition}
Let $(F_t)_{t \geq 0}$ be a c\`adl\`ag,
 adapted, nonnegative, uniformly bounded and nonincreasing process. We assume that there exists
$a >0$ such that for all $t \geq 0$, $X_t = 
X_{g^{[a]}}$ on the set $\{t \geq g^{[a]}\}$, we
define $F_{\infty}$  as the limit of $F_t$ for $t$ going to infinity, and we suppose
 that $0 < \mathcal{Q} [F_{\infty}] < \infty$. Then, for all $t \geq 0$:
$$0< \mathbb{P} [F_t] < \infty,$$
and one can then define a probability measure $\mathbb{Q}_t$ on $(\Omega, \mathcal{F})$ by
$$\mathbb{Q}_t := \frac{F_t}{\mathbb{P}[F_t]} \, . \mathbb{P}.$$
Moreover, the probability measure:
$$\mathbb{Q}_{\infty} := \frac{F_{\infty}}{ \mathcal{Q} [F_{\infty}]} \, . \mathcal{Q}$$ 
is the weak limit of $\mathbb{Q}_t$ in the sense of penalisations,  i.e. for all $s \geq 0$, 
and for all events $\Lambda_s \in \mathcal{F}_s$,
$$\mathbb{Q}_t [\Lambda_s] \underset{t \rightarrow \infty}{\longrightarrow} \mathbb{Q}_{\infty} [\Lambda_s].$$
\end{proposition}
\noindent
This penalisation result applies, in particular, to the power $2r$
of a Bessel process of dimension $2(1-r)$, for any $r \in (0,1)$ (in this case, the parameter $\beta$ 
is equal to $(1/r)-2$). For $r=1$, $(X_t)_{t \geq 0}$ is a reflected Brownian motion, and our result is very
 similar to the penalisation theorem stated in \cite{NRY} for the Brownian motion.

\providecommand{\bysame}{\leavevmode\hbox to3em{\hrulefill}\thinspace}
\providecommand{\MR}{\relax\ifhmode\unskip\space\fi MR }
\providecommand{\MRhref}[2]{%
  \href{http://www.ams.org/mathscinet-getitem?mr=#1}{#2}
}
\providecommand{\href}[2]{#2}

\end{document}